\documentclass[12pt,twoside]{article}
\usepackage{amsmath}
\usepackage{amsthm}
\usepackage{amssymb}
\usepackage[margin=1in]{geometry}
\usepackage{graphicx}
\usepackage{ifpdf}

\ifpdf
 \DeclareGraphicsExtensions{.pdf}
\else
 \DeclareGraphicsExtensions{.eps}
\fi

\title{On increasing stability in the two dimensional  inverse source scattering problem with many frequencies}
\author{ Mozhgan Nora Entekhabi and Victor Isakov}

\date{December 22, 2017}

\begin{document}
\maketitle
\newtheorem{theorem}{Theorem}[section]
\newtheorem{lemma}[theorem]{Lemma}
\newtheorem{corollary}[theorem]{Corollary}
\newtheorem{definition}[theorem]{Definition}
\newtheorem{proposition}[theorem]{Proposition}

\[
Department\; of\; Mathematics, Statistics,\; and\; Physics
\]
\[
Wichita\; State\; University
\]
\[
Wichita,\; KS\; 67260-0033,\; U.S.A.
\]
\[
e-mail:\;\; entekhabi@math.wichita.edu\;\;victor.isakov@wichita.edu\;\;
\]

{\bf Abstract}

In this paper, we will study increasing stability in the  inverse source problem for the Helmholtz equation in the plane when the source term 
is assumed to be compactly supported in a bounded domain $\Omega$ with sufficiently smooth boundary. Using the Fourier transform in the frequency domain,  bounds for the Hankel functions and for scattering solutions in the complex plane, improving bounds for the analytic continuation, and exact observability for wave equation led us to our goals which are a sharp 
uniqueness and increasing stability estimate with larger wave numbers interval.  

\vspace{10pt}


\vspace{10pt}


\section{Introduction.}

The inverse source problem seeks for the right hand side of a partial differential equation from boundary data. It has important applications, for example, in acoustical and  biomedical/medical imaging, antenna synthesis, geophysics, and material science \cite{ABF}, \cite{B}. 
 In this paper, we are interested in uniqueness and stability in this inverse problem for the  Helmholtz equation when the source is  supported in a bounded domain $\Omega$. As we know \cite{I} the inverse source problem does not have a unique solution at a single or at finitely many wave numbers. On the other hand, if we use  all wave numbers in $(0,K)$ one can regain uniqueness. Our goal is to establish  uniqueness for the source from the Cauchy data on any open non empty part of the boundary for arbitrary positive $K$ and increasing stability when the Cauchy data are given on the whole boundary and $K$ is getting large. 
 
Increasing stability results were obtained in \cite{BJT} by using the spatial Fourier transform. In \cite{CIL} more general and sharp results were obtained in $\mathbb{R}^3$ by the temporal Fourier transform, with a possibility of handling spatially variable coefficients. In \cite{IL} an inverse source problem with attenuation with many frequencies was studied. New difficulties in the two-dimensional case are due to the  absence of the Huygens principle and a more complicated fundamental solution compared with  the  three-dimensional case \cite{CIL}. In \cite{LY} the authors studied increasing stability for the source when  $\Omega$ is a disk.

In this paper we handle a general bounded domains $\Omega$ in the plane. The two-dimensional case has numerous applications to surface vibrations and due to the lack of the Huygens principle and relatively complicated fundamental solutions it highlights difficulties of more general situations and suggests a ways to overcome them. As in \cite{CIL}, we  use the Fourier transform in time to reduce our inverse source problem to finding the initial data in the hyperbolic value problem by  the lateral Cauchy data. We achieved this goal by obtaining improving with growing $K$ sharp bounds of the analytic continuation of the data from $(0,K)$ onto $(0,+\infty)$ and combining them with known optimal Lipschitz stability of the initial data from the lateral boundary Cauchy data for the wave equation. These bounds are obtained by using certain properties of the Hankel function instead of the exponential function as in
\cite{CIL}. The main new difficulties are handled in section 4 where we use a sharp
time decay for the plane wave equation and imposed higher (compared with \cite{CIL}) regularity assumptions on the source term to obtain needed a priori bounds on the decay of scattering solutions for larger wave numbers.
 
This paper is organized as follows. In section 2 we prove uniqueness in the inverse source problem for the Helmholtz equation. Here, we use the  analyticity of the solution $u(x,k)$ with respect to $k$ and sharp uniqueness of the continuation results for the wave equation. In section 3 by using known estimates of the Hankel functions we adjust the methods of \cite{CIL} to get exponential (with respect to the wave numbers) bounds of integrals of certain norms of  the scattering solutions for complex wave numbers. Also we recall explicit bounds of the harmonic measure of the interval $(0,K)$ in a sector of the complex plane $k=k_1 +ik_2$ and consequently an explicit bound of the analytic continuation of $u(,k)$ from interval $(0,K)$ onto the real axis. Finally, in section 4 we use the known integral representation to get decay rates for solutions to the initial value problem for the wave equation in the plane to justify the temporal Fourier transform of its solution and to link it to the scattering solutions. Moreover, we use the known recursion relations for the Hankel functions, the integration by parts and higher regularity of the source term to derive the needed time decay of scattering solutions with respect to the wave numbers.
This was easy to do in \cite{CIL} due to the Huygens principle in the three-dimensional case. Then we combine all preliminary results to derive the increasing stability bound. 

Let the radiating wave field $u(x,k)$ solve the scattering problem 
\begin{equation}
\label{PDE} 
(\Delta +k^2)u = -f_1 -ikf_0 \quad \text{in} \quad \mathbb{R} ^2,
\end{equation}
\begin{equation}
\label{radiation}
\lim r^{1/2}( \partial_{r}u-iku)=0 \quad \text{as}\quad r=|x| \rightarrow +\infty.
\end{equation}
Both $f_0, f_1 \in L^2(\Omega)$ are assumed  to be real valued and to have $supp f_0,suppf_1 \subset \Omega$ where $\Omega $ is a bounded domain with the boundary  $\partial \Omega \in C ^2 $. 
 
Our  goal is to study  uniqueness and stability  of functions $f_0,f_1$ from the data
\begin{equation}
\label{Cauchy}
u=u_0, \, \partial_ \nu u =u_1 \, \text{on} \, \Gamma, \, \text{when} \quad  K_* < k< K,
\end{equation}
where $\Gamma$ is an non empty open subset of $\partial \Omega $ with outer unit normal $\nu$ and $0< K_* < K$.

Our main result is the following

\begin{theorem} 
\label{maintheorem}
Let $\parallel f_0 \parallel _{(4)} ^2 (\Omega) +\parallel f_1 \parallel _{(3)} ^2 (\Omega) \leq M$, $1 \leq M$,
and $\delta<|x-y|, x\in\partial \Omega, y\in supp f_0 \cup supp f_1$ for some 
positive
$\delta$. 

Then there exist a constant $C=C(\Omega,\delta)$ such that
\begin{equation}
\label{Istability}
\parallel f_1 \parallel _{(0)} ^2 (\Omega) +\parallel f_0 \parallel _{(1)} ^2 (\Omega) \leq   C\Big(\epsilon ^2+\frac{M^2}{1+K^{\frac{2}{3}}E^{\frac{1}{4}}} \Big)
\end{equation}
for all $u \in H^2 (\Omega)$ solving (2.1), (2.2) with $1<K$. Here 
\begin{equation}
\label{epsilon}
\epsilon^2=\int _{0}^{K} \Big ( \omega ^2 \parallel u(,\omega) \parallel_{(0)}^{2} (\partial \Omega) +\parallel \nabla u(,\omega) \parallel_{(0)}^{2} (\partial \Omega)  \Big)d \omega, \quad 0<E=-\ln  \epsilon.
\end{equation}
\end{theorem}

Observe that \eqref{Istability} is a conditional improving stability estimate for the source term by the Cauchy data \eqref{Cauchy}. While the logarithmic term $E$ is necessary at any fixed upper
bound of the  wave numbers $K$ because the Helmholtz equation is of the elliptic type, with increasing $K$ this term goes to zero and the estimate is getting closer to the optimal Lipschitz one. 

 The following well known integral representation holds   
\begin{equation}
\label{uH}
u(x,k)= \frac{-i}{4} \int_{\Omega} H_{0} ^{(1)}(k|x-y|) (f_1(y)+ikf_0(y))dy,
\end{equation}
where $ H_{0} ^{1}(z)= \frac{1}{\pi i} \int _{1+i\infty}^{1} e^{izs} (s^2  -1)^{-1/2} ds $, for $0<\Re z$, is the Hankel function of the first kind \cite{W}. It is also can be defined as  
\begin{equation*}
H_{0}^{(1)}(z)=J_0(z)+i Y_0(z),
\end{equation*}
where 
\begin{equation*}
J_0(z)=\sum _{m=0}^{\infty} \frac{(-1)^m  (\frac{1}{2} z)^{2m}}{(m!)^2},\;
Y_0(z)=2\{\gamma +\log(\frac{1}{2}z) \}  J_0 (z) -2 \sum _{m=1}^{\infty} \frac{(-1)^m  (\frac{1}{2} z)^{2m}}{(m!)^2}\{1+ \frac{1}{2}+...+\frac{1}{m}     \},
\end{equation*}
and $\gamma = 0.5772157...$ is the  Euler's constant. Using the recurrence formula $H^{(1)}_{0} (z)= (2/z)H^{(1)}_{1} (z)-H^{(1)}_{2} (z)$, and the Weber crude inequality
\begin{equation*}
|H_{\nu}^{(1)} (z)|\leq |\big(  \frac{2}{\pi z}   \big)^{\frac{1}{2}} e^{i(z-\frac{1  }{2}\nu \pi - \frac{1}{4} \pi)}|. \Big ( 1- \frac{\nu- \frac{1}{2}}{2r} \Big) ^{(-\nu- \frac{1}{2})} 
\end{equation*}
when $\nu-\frac{1}{2}>0$ (see \cite{W}, chapter VII, p.210),
the following well known bound for Hankel function of the first kind  is obvious: 
\begin{equation}
\label{H0bound}
|H_{0}^{(1)} (z)|  \leq |\big(  \frac{2}{\pi z}   \big)^{\frac{1}{2}} e^{i(z-\frac{3 \pi}{4})}|. \Big ( \frac{2}{r}  \big( 1- \frac{1}{4r} \big)^{-\frac{3}{2}}  + \big ( 1- \frac{3}{4r} \big)^{-\frac{5}{4}} \Big ).
\end{equation}
when $ 0< \Re z$, $ \frac{3}{4}<r$, $r=|z|$ .

 In what follows $C$ denote generic constants depending on the domain $\Omega $ and on $\delta$, any additional dependence will be indicated in parenthesis.  $||u ||_{(l)} (\Omega)$ denotes the standard norm in Sobolev space $H^{l}(\Omega)$.
 
\section{Uniqueness of source identification}

In this section we will demonstrate the uniqueness of the source under minimal assumptions.

\begin{theorem}
\label{uniqueness}
 Let $u$ be a solution to the scattering problem \eqref{PDE}, \eqref{radiation} with $f_0\in H^1(\Omega), f_1\in L^2(\Omega)$.
 If the Cauchy data $u_0=u_1 =0$ on $\Gamma$ when $k\in (K_*,K) $, then $f_0=f_1=0$ in $\Omega$.
\end{theorem} 
 
\begin{proof}
Denote by $U_0$ the solution to the following hyperbolic problem 
\begin{equation}
\label{U0}
\partial ^{2}_{t}U_0 -\Delta U_0 = 0 \quad \text{on} \quad \Omega \times(0,\infty),
\end{equation}
\begin{equation*}
U_0 =-f_0,\, \partial _t U_0 =f_1 \quad \text{on} \quad \Omega \times \{0\,\}, \, U_0=0 \quad \text{on} \quad \partial \Omega \times (0,+\infty).
\end{equation*}
Under our assumptions, there is a unique solution to the problem \eqref{U0} with
 \begin{equation}
 \label{energy}
 \parallel U_0 (,t)\parallel _{(1)}(\Omega) + \parallel \partial _t U_0 (,t)\parallel _{(0)}(\Omega) 
 \leq C (1+t)(\parallel f_0\parallel _{(1)}(\Omega) +\parallel f_1\parallel _{(0)}(\Omega)  ),
 \end{equation}
 which easily follows from the standard energy integral for the wave equation.

Now let 
\begin{equation*}
u^*(x,k)= \int _{0}^{\infty} U_0(x,t)e^{ikt} dt
\end{equation*}
 Due to \eqref{energy}, the function $u^* (x,k)$ is well defined and analytic with respect to $k=k_1+ik_2, k_2>0$ . Applying the integration by parts and using standard properties of the Fourier-Laplace transform we conclude that
\begin{equation}
\label{u*}
(\Delta + k^2 ) u^* = -f_1-ikf_0 \quad \text{in} \quad \Omega,\quad u^* =0 \quad \text{on} \quad \partial \Omega.
\end{equation}

Due to our assumption, the function $u$ solves the same Dirichlet problem for $\Delta +k^2$ when $0<k_1, 0<k_2$. Indeed, $u$ solves the homogeneous Helmholtz 
equation in $\mathbb{R}^2\setminus \bar\Omega$ and has zero Cauchy data on $\Gamma$. By the uniqueness in the Cauchy problem for elliptic equations 
$u=0$ on $\mathbb{R}^2\setminus \bar\Omega$ and hence on $\partial\Omega$ provided $K_*<k<K$. As follows from the integral representation \eqref{uH} and \eqref{H0bound}
the function $u( ;k)$ is (complex)analytic when $0<\Re k$, hence $u( ;k)=0$ on $\partial \Omega$ provided
$0< \Re k$. Since $k_2>0$, the solution of \eqref{u*} is unique, hence $u=u^*$ on $\Omega $ (see section 4). Consequently, we obtain $u^* =u =0$, $ \partial _\nu u^*= \partial _ \nu u = 0 $ on $\Gamma$. Since $u^*$ is an analytic function, we can conclude that $u^* =0, \partial_ \nu u^* =0  $ on $\Gamma$ for all $k=k_1+ i k_2$ with $k_2>0$. Due to the uniqueness of the inversion of the Fourier-Laplace transform 
\begin{equation*}
\partial _ \nu U_0=0 \quad \text{on} \quad \Gamma \times (0,\infty). 
\end{equation*}
Due to the uniqueness in the lateral Cauchy problem for the wave equation  \eqref{U0} with  the Cauchy data on $\Gamma \times (0,+\infty )$ (Holmgren-John theorem \cite{I}, Section 3.4), we can conclude that $U_0 =0$ on $\Omega \times (T,+\infty)$ for some positive $T$. Hence from the uniqueness in the backward initial boundary value problem for the hyperbolic equation \eqref{U0}  in $\Omega \times (0,T)$ with zero boundary data on $\partial \Omega \times (0,T)$ and initial data at $\Omega \times \{T\}$ we conclude that $U_0 =0  $ on $\Omega \times (0,T)$. So $-U_0(,0) = f_0 =0, \partial _t U(,0)=f_1=0 $ on $\Omega$ which finishes the proof of uniqueness.

\end{proof}

Observe that this proof can be
immediately generalized onto arbitrary elliptic equations of the second order with variable coefficients which coincides with the Helmholtz equation outside $\Omega$.
Moreover, with some
modifications it can be used to demonstrate the uniqueness
of the source in the more complicated case when the equation \eqref{PDE} holds true only in $\Omega$ and the radiation condition is replaced by some homogeneous elliptic boundary condition
on $\partial\Omega$.

\section{Increasing stability of the continuation to high frequencies}

We introduce the functions
$I_1(k), I_2(k)$ as follows: 
\begin{equation}
\label{I1}
 \int _{-\infty} ^{+\infty} \omega ^2 \parallel u(,\omega) \parallel^{2}_{(0)}(\partial \Omega) d\omega = I_1 (k)+ \int _{k<|\omega|}  \omega ^2 \parallel u(,\omega) \parallel^{2}_{(0)}(\partial \Omega) d\omega,
\end{equation}
\begin{equation}
\label{I2}
 \int _{-\infty} ^{+\infty} \omega ^2 \parallel \nabla u(,\omega) \parallel^{2}_{(0)}(\partial \Omega) d\omega = I_2 (k)+ \int _{k<|\omega|}  \omega ^2 \parallel \nabla u(,\omega) \parallel^{2}_{(0)}(\partial \Omega) d\omega.
\end{equation}

The purpose of this section is to obtain some upper bounds for these functions of $k$ in a sector of the complex plane. These bounds are needed to get bounds of the analytic continuation of 
$I_1(k), I_2(k)$ from $(0,K)$
(where they are less than $\epsilon^2$) onto larger interval of wave numbers in order to bound the left hand sides of \eqref{I1},\eqref{I2}. Due to the Parseval's identity these left hand sides are norms of the lateral boundary Cauchy
 data for the wave equation which are needed to bound the initial data for this equation, or the source term for the Helmholtz equation \eqref{PDE}.  

Using \eqref{uH}, 
\begin{equation*}
I_1= 2\int _0 ^ k \omega ^2 \int _{ \partial \Omega} \big(\int _{\Omega} \frac{-i}{4} (f_1(y)+i\omega f_0(y))H_{0} ^{(1)}(\omega|x-y|) dy \big) 
\end {equation*}
\begin{equation}
\label{I1i}
\big(   \int _{\Omega} \frac{i}{4} (f_1(y)-i\omega f_0(y)) \overline{H_{0} ^{(1)}}(\omega|x-y|) dy  \big)d\Gamma (x)d\omega,
\end{equation}
\begin{equation*}
I_2= 2\int _0 ^ k \omega ^2 \int _{ \partial \Omega} \big(\int _{\Omega} \frac{-i}{4}  (f_1(y)+i\omega f_0(y))\nabla _x H_{0} ^{(1)}(\omega|x-y|) dy \big)
\end{equation*}
\begin{equation}
\label{I2i}
\big(   \int _{\Omega} \frac{i}{4} (f_1(y)-i\omega f_0(y)) \nabla_x \overline{H_{0} ^{(1)}}(\omega|x-y|) dy  \big)d\Gamma (x)d\omega.
\end{equation}
where $k>0$ and $ \overline{H_{0} ^{(1)}}(z)=J_0(z)-i Y_0(z)$, provided $z\in(0,\infty)$.

In the increasing stability estimates we utilize the  norm
\eqref{epsilon} of the Cauchy data which is $I_1(K)+ I_2(K)$. Since for $k\in (K,\infty)$ the data would be unknown, the truncation level of $k$ in $I_1(k), I_2(k)$ will keep balance between the value $\epsilon$ and the unknown information.
The integrands in \eqref{I1i},
\eqref{I2i}  are analytic functions of $\omega$, hence $I_{1}(k), I_{2}(k)$ are  analytic functions of $k$ in $\mathbb{C}\backslash(-\infty,0]$ .

 \begin{lemma}
 \label{Ibounds} 
 Let $f_1 \in H^1 (\Omega), f_0 \in H^1(\Omega)$ and $supp f_0, suppf_1  \subset \Omega $. Then
\begin{equation}
\label{boundI1}
|I_1 ( k) | \leq  \frac{\pi}{2} |\partial \Omega|d\Big(  \frac{1}{3} |k|^3  \parallel f_1\parallel_{(0)}^2  (\Omega)+   
\frac{1}{5} |k|^5  \parallel f_0 \parallel_{(0)}^2  (\Omega)  \big)\frac{e^{2d|k_2|}}{k_{1}},
\end{equation}
 \begin{equation}
 \label{boundI2}
|I_2 ( k) | \leq  \frac{\pi}{2} |\partial \Omega|d\Big(  \frac{1}{3} |k|  \parallel f_1\parallel_{(1)}^2  (\Omega)+  
 \frac{1}{5} |k|^3 \parallel f_0 \parallel_{(1)}^2  (\Omega)  \big)\frac{e^{2d|k_2|}}{k_{1}},
\end{equation}
 where $|\partial \Omega|$ is the length of $\partial \Omega$, $d=sup|x-y|$ over $x, y \in \Omega$.
 \end{lemma}
 
 \begin{proof}
  The Hankel function $H^{(1)}_{0}(z)$ can be bounded  by a simple change of variable as follows (see \cite{LY}, Lemma 3.2),
\begin{equation*}
|H^{(1)}  _{0}(z)| \leq \frac{e^{|Imz|}}{(Rez)^{1/2}}, \quad |\overline {H^{(1)}  _{0}(z)}| \leq \frac{e^{|Imz|}}{(Rez)^{1/2}}
\end{equation*}
for all $z$ with $0< \Re z$, so we can conclude that
$$
|H^{(1)}  _{0}(k|x-y|)| \leq \frac{e^{|k_2||x-y|}}{(k_1|x-y|)^{1/2}}.$$
 Now, 
  using the parametrization $\omega =ks, s\in (0,1)$, in the line integral we obtain the  inequality:
 \begin{equation*}
 |I_1 (k) |\leq \frac{1}{8} \int_{0}^{1}|k|^{3} s^2 \Big( \int _{\partial \Omega } \Big( \int _{\Omega} \big(|f_1 (y)|+ s|k||f_0 (y)| \big)   \frac{e^{|k_2||x-y|}}{(k_1|x-y|)^{1/2}}dy \Big)^2 d\Gamma(x) \Big)ds
 \end{equation*} 
 \begin{equation*}
 \leq \frac{1}{4}  \int_{0}^{1}|k|^{3} s^2 \int _{\partial \Omega } \Big( \int _{\Omega} \big(|f_1 (y)|^2+ s^2 |k|^2|f_0 (y)|^2 \big)dy\Big)\Big(     \int_{\Omega} \frac{e^{2|k_2 |d}}{k_{1} |x-y|}dy \Big) d\Gamma(x)ds, 
 \end{equation*}
 where the Schwartz inequality is used for the integral with respect to $y$. 
  Using the polar coordinate $\rho=|y-x|$ centered at $x$ we will obtain
 \begin{equation*}
|I_1(k)|\leq \frac{1}{4}  |k|^{3}\int_{0}^{1} s^2 \int _{\partial \Omega } \Big( \int _{\Omega} \big(|f_1 (y)|^2+ s^2 |k|^2|f_0 (y)|^2 \big)dy\Big)\Big(2\pi     \int_{0}^{d} \frac{e^{2|k_2|d}}{k_{1} }d\rho \Big) d\Gamma(x)ds 
 \end{equation*}
 Integrating with respect to $s, x,$ and $\rho$, we complete the proof of \eqref{boundI1}.\\
 
 To obtain the  bound \eqref{boundI2}, we will use that 
 \begin{equation*}
 \nabla _{x} H^{(1)}_0(k|x-y|) = -\nabla _{y} H^{(1)}_0(k|x-y|).
 \end{equation*}
Then \eqref{boundI2} can be proved integrating by parts in the integral over domain $\Omega$ and by argument similar to the case of \eqref{boundI1}.
 \end{proof}
 
The rest of this section will provide argument which is essential in order to connect $I_1  (k) $ and $I_2  (k)$ for $k\in [K,\infty)$ to $I_1(K),
I_2(K)$. Let us introduce the sector $S=\{ k \in \mathbb{C}:|$arg$\, k| < \frac{\pi}{4}    \}$. It is easy to see that $ |k|\leq  k_1   $ for any $k \in S$. Hence we conclude

\begin{equation*}
|I_1 (k) e^{-2(d+1)k}|\leq C \Big(  |k _1| ^2 \parallel f_1 \parallel_{(0)}^2  (\Omega)+    |k_1| ^{4} \parallel f_0 \parallel _{(0)}^2  (\Omega)  \Big ) e^{-2k_1} \leq C M^2,
\end{equation*}
when $ \parallel f_1\parallel_{(0)}^2  (\Omega)+ \parallel f_0 \parallel_{(0)}^2  (\Omega)\leq M^2$.
Using \eqref{epsilon}, it is obvious that
\begin{equation*}
|I_1 (k) e^{-2(d+1)k}|\leq  \epsilon ^2 \text{ \, on \, } [0,K],  
\end{equation*}
The following will provide a connection between $I_1(k)$ and $\epsilon$ for any $k\in [K,+\infty]$.

Let $\mu (k) $ be the harmonic measure of the interval $[0,K]$ in $S\backslash [0,K] $, then 
\begin{equation}
\label{I1epsilon}
|I_1 (k) e^{-2(d+1)k}|\leq C \epsilon ^{2\mu (k)} M^2 ,  
\end{equation}
where  $K<k<\infty$.

With a similar argument we obtain 
\begin{equation}
\label{I2epsilon}
|I_2 (k) e^{-2(d+1)k}|\leq C \epsilon ^{2\mu (k)} M^2.
\end{equation}

In order to find a lower bound for $\mu (k)$ we need the following lemma .

\begin{lemma} 
 \label{Lemma32}
Let $\mu (k)$ be the harmonic measure of $[0,K]$ in $S \backslash [0,K]$.
 Then  
\begin{equation*}
\begin{cases}
    \frac{1}{2} \leq \mu(k)  \qquad  \qquad \qquad \qquad \text{if} \qquad 0<k<2^{\frac{1}{4}}K ,\\
   \frac{1}{\pi} \big( (\frac{k}{K})^4   - 1\big) ^{- 1/2} \leq \mu (k) \, \quad \text{if} \qquad  \quad  2^{\frac{1}{4}}K<k.
   \end{cases}
\end{equation*}
\end{lemma}

A proof  can be found in \cite{CIL}.

\section{Increasing stability for the inverse source problem}

We consider the hyperbolic initial value problem
\begin{equation*}
(\partial ^2 _ t  - \Delta )U=0 \text{ in } \mathbb{R}^2\times (0,\infty),
\end{equation*}
\begin{equation}
\label{initialhyperbolic}
U(,0)=-f_0, \quad \partial_t U (,0)=f_1 \text{ on }  \mathbb{R}^2.
\end{equation}
 Its solution $U$ has the following well-known integral representation \cite{J}:
 \begin{equation}
 \label{Uintegral}
U(x,t)= \frac{1}{2\pi}  \iint_{ |x-y|<t} \frac{f_1 (y)}{\sqrt[]{t^2 - r^2}}dy + 
\partial _t \Big(   \frac{1}{2\pi} \iint_{|x-y|<t} \frac{-f_0 (y)}{\sqrt[]{t^2 - r^2}}dy \Big), 
\end{equation}
where $r=|x-y|$. We define $U(x,t)=0 $ when $t<0$. We claim that the solution of \eqref{PDE}, \eqref{radiation} coincides with the Fourier transform of $U(x,t)$ with respect to $t$. In order to show it, the following steps are essential.

For $k_2\geq 0$, we define 
\begin{equation}
\label{u1}
u_1(x,k)=\int_{0}^{\infty} U(x,t)e^{ikt} dt.
\end{equation}
Observe that, $U(x,t)=0$ when $t< |x|-d_1$, where $d_1= \underset{y \in \Omega}{ sup|y|}$, due to the unit speed of the propagation.  

First let us assume that $f_1=0$. Using the triangle inequality, $|x-y| \leq |x|+|y| \leq |x|+d_1$, for $y\in \Omega $, so we conclude that $t^2-|x-y|^2\geq t^2 -(|x|+d_1)^2$. Let   $T=|x|+d_1+1$ and consider the following representation 
\begin{equation}
\label{u1split}
u_1 (x,k)= \int _{0}^{T} U(x,t)e^{ikt} dt + 
\int _{T}^{\infty} \partial _t \Big(   \iint_{ \Omega} \frac{-f_0 (y)}{\sqrt[]{t^2 - r^2}}dy \Big)e^{ikt} dt.
\end{equation}
Due to the properties of $U$, in particular to \eqref{energy}, the first integral of the right hand-side of \eqref{u1split} is well defined.

If $t>T$, the domain of integration of the second integral in \eqref{u1split} with respect to $y$, is $\Omega$ and hence independent of $t$, 
function $\frac{-f_0 (y)}{\sqrt[]{t^2 -r^2}}$ is integrable with respect to $y$ and when $0\leq k_2$ we have 
\begin{equation}
\label{boundk}
|\frac{\partial }{\partial t}\big (\frac{ f_0 (y)}{(t^2 - r^2)^{\frac{1}{2}}} \big ) e^{ikt}       | = |\frac{te^{ikt} f_0 (y)}{(t^2 - r^2)^{\frac{3}{2}}}|\leq \frac{ t |f_0 (y)|}{(t^2 -(|x|+ d_1 )^2)^{\frac{3}{2}}} ,
\end{equation}
which is integrable function of $t$ over $(T,+\infty)$. When $0\leq k_2$, using the Lebesgue's Dominated Convergence Theorem  we yield
\begin{equation}
\label{partialt}
\int _{T}^{\infty} \partial _t \Big(   \iint_{ \Omega} \frac{-f_0 (y)}{\sqrt[]{t^2 - r^2}}dy \Big)e^{ikt} dt=   \int _{T}^{\infty}    \iint_{ \Omega} \frac{t e^{ikt} f_0 (y)}{(t^2 - r^2)^{\frac{3}{2}}} dy dt.   
\end{equation}

Combining these two cases, we conclude that $u_1$ is well-defined for $0\leq k_2$. 

Calculating the $L^2$ norm of $u_1(x,k)$ when $k_2>0$, we conclude that $u_1$ exponentially decays when $|x|$ is getting large. Indeed if $ d_1< R$, we have
\begin{equation*}
  \int _{B(R+2)\backslash B(R)}|u_1|^2  dx = \int _{B(R+2)\backslash B(R)}|   \int_{|x|-d_1}^{\infty} U(x,t)e^{ikt} dt |^2 dx \leq
\end{equation*}
\begin{equation*}
 \int _{B(R+2)\backslash B(R)}  (\int _{R-d_1}^{\infty}|U(x,t)|^2e ^{-k_2 t}dt  \int _{R-d_1}^{\infty}e ^{-k_2 t}dt )dx  \leq
\end{equation*}
\begin{equation}
\label{u1bound}
\frac{1}{k_2}e^{-k_2(R-d_1)}
\int_ {R-d_1}^{\infty}\Big(   \int _{{\mathbb R}^2}|U(x,t)|^2 dx\Big) e^{-k_2t}dt
 \leq C (k_2)\|f_0\|^2_{(0)}(\Omega)  e^{-k_2R},
\end{equation}
where we used the H\"lder inequality in the integral with respect to $t$ and \eqref{energy}.

A similar bound can be obtained for $\nabla u_1$.

If $f_0=0$, integrating by parts  and using \eqref{Uintegral} we obtain 
$$
u_1(x,k)=-\frac{1}{k} \int _{0}^{\infty} \partial_t U(x,t) e^{ikt} dt.     
$$
 Utilizing again \eqref{Uintegral} we handle this case exactly as before. From the linearity of $U(x,t)$ with respect to $f_0$ 
and $f_1$  we obtain the general case.

Using \eqref{u1} and the integration by parts, we yield
\begin{equation*}
(\Delta +k^2)u_1=\int_{0}^{\infty} (\partial_{t}^{2} +k^2)Ue^{ikt} dt=-f_1 -ikf_0  \quad \text{in}  \quad \mathbb{R}^2.
\end{equation*}

Consider the function $v=u_1-u$. It is easy to see that $(\Delta +k^2)v=0$ and also due to the decay of the Hankel function for large $|x|$ when $k_2>0$ and 
\eqref{u1bound}, $v(x)$ , $\nabla v(x)$ decay exponentially  for large $|x|$ when $k_2>0$. Let us choose a cut off function $\chi \in C^{\infty}(\mathbb{R}^2)$, $|\nabla \chi| \leq C $ which is $\chi =1$ when $|x|< R$, $ \chi =0$ when $R+1<|x|$ and $0\leq  \chi \leq 1$. We have
\begin{equation*}
\int _{B(R+1) } (\Delta  (\chi v) + k^2 (\chi v) ) \chi \overline{v}=\int _{B(R+1)  } \big (  \chi (\Delta  +k^2)v +2\nabla \chi \nabla v + \Delta \chi v        \big) \chi \overline{v}=
\end{equation*}
\begin{equation}
\label{product}
 \int _{B(R+1) \backslash B(R) } \big (2\nabla \chi \nabla v + \Delta \chi v        \big) \chi \overline{v}.
\end{equation}
On the other hand, from the Green formula
\begin{equation}
\label{Green}
\int _{B(R+1) } (\Delta  (\chi v) + k^2 (\chi v) ) \chi \overline{v}= - \int _{B(R+1)} \nabla \chi v.\nabla \chi \overline{v}+ k^2 \int _{B(R+1) } \chi ^2 v \overline{v}.
\end{equation}
Taking the imaginary part of \eqref{product}, \eqref{Green} when $k_2 >0 $, we obtain  
\begin{equation*}
\mathfrak{Im} \int _{B(R+1)  \backslash B(R)} \big( 2\nabla \chi \nabla v + \Delta \chi v       \big) \chi \overline{v} =2k_1 k_2 \int _{B(R+1)  } \chi ^2 |v|^2 .
\end{equation*}
Due to an exponential decay of $v$ for large $R$, the limit of left hand side is $0$. Hence  $v=0$ on $\mathbb{R}^2$  for $k_2>0$, and $u_1(,k)=u(,k)$ when $0<k_2$.
From \eqref{uH} we can see that $u( k)$ is complex analytic in $\{k: 0< k_1\}$.
Using again the Lebesgue Dominated Convergence Theorem and \eqref{u1split}, \eqref{boundk},
and \eqref{partialt} we conclude that $u_1( ,k)$ is continuous with respect to $k_2, 0\leq k_2$. Finally, passing to the limit in     $u_1(,k)=u(,k)$ when $k_2\rightarrow 0$ we obtain 
this equality when $k>0$, or,
in view of \eqref{u1},
\begin{equation}
\label{uF}
u(x,k)= \int_{0}^{\infty} U(x,t)e^{ikt} dt, \;\text{when}\; k>0.
\end{equation}

\begin{lemma} 
\label{Lemma41}
Let function $u$ be a solution to the forward problem \eqref{PDE}, \eqref{radiation} with $f_1 \in H^3 (\Omega)$ and $f_0 \in H^4 (\Omega)$, $supp f_1, supp f_0 \subset \Omega $.

 Then 
\begin{equation*}
\int_{k<|\omega|} \omega ^2 \parallel u(,\omega) \parallel^2 _{(0)} (\partial \Omega) d\omega+ \int    _{k<|\omega|}    \parallel \nabla u(,\omega) \parallel^{2}_{(0)}(\partial \Omega) d\omega 
\leq 
\end{equation*}
\begin{equation}
\label{uomega}
C k ^{-1}\Big( \parallel f_0 \parallel ^2  _{(4)} (\Omega)  +\parallel f_1 \parallel ^2  _{(3)} (\Omega)  \Big).
 \end{equation}
\end{lemma}

\begin{proof}

Using the following well known formula for the Hankel function \cite{W}, p 74, 5): 
\begin{equation*}
z\frac{dH_{\nu}^{(1)}(z)}{dz}+\nu H_{\nu}^{(1)}(z)=zH_{\nu-1}^{(1)}(z), 
\text{or} 
\frac{d}{dz}(z^{\nu}H_{\nu}^{(1)}(z))=z^{\nu}H_{\nu-1}^{(1)}(z), \nu=1,2,...
\end{equation*}
we yield
\begin{equation}
\label{Hrecursion}
\omega^{-1}
\partial_r(r^{\nu}H_{\nu}^{(1)}(r\omega))=r^{\nu}H_{\nu-1}^{(1)}(r\omega), \nu=1,2,...
\end{equation}
Let $ d=diam\Omega$, using
the integral representation  \eqref{uH}, the polar coordinates in the plane $\theta$, $r=|y-x|$ originated at $x\in \partial \Omega$,
\eqref{Hrecursion}, and the integration by parts, we obtain: 
 \begin{equation*}
16 \int_{k<|\omega|} \omega ^2 \parallel u(,\omega) \parallel^2 _{(0)} (\partial \Omega) d\omega
= 
\end{equation*}
\begin{equation*}
 \int _{k}^{\infty}  \omega ^2 \int _{ \partial \Omega} \Big | \int _{0}^{2\pi} d\theta \int_{ \delta} ^{d} H_{0} ^{(1)}(r\omega) (f_1 +i\omega f_0 )rdr     \Big|^{2}d\Gamma(x)d\omega=
\end{equation*}
\begin{equation*}
  \int _{k}^{\infty}  \omega ^2 \int _{ \partial \Omega} \Big | \int _{0}^{2\pi} d\theta \int_{ \delta} ^{d}    \frac { 
 \partial_r (r H_{1} ^{(1)}(r\omega) )}{ \omega  }( f_1 +i \omega f_0 )dr     \Big|^{2}d\Gamma(x)d\omega=
 \end{equation*}
  \begin{equation*}
  \int _{k}^{\infty}  \omega ^2 \int _{ \partial \Omega} \Big | \int _{0}^{2\pi} d\theta \int_{ \delta} ^{d} r^2  H_{1} ^{(1)}(r\omega) \frac{1}{r} \partial_r(\frac{1}{\omega} f_1 + i f_0 )dr     \Big|^{2}d\Gamma(x)d\omega=
\end{equation*}
\begin{equation*}
 \int _{k}^{\infty}  \omega ^2 \int _{ \partial \Omega} \Big | \int _{0}^{2\pi} d\theta \int_{ \delta} ^{d} \partial_r (r^2  H_{2} ^{(1)}(r\omega) )\frac{1}{r}\partial_r(\frac{1}{\omega^2}f_1 +\frac{i}{\omega} f_0 )dr     \Big|^{2}d\Gamma(x)d\omega=
 \end{equation*}
 \begin{equation*}
  \int _{k}^{\infty}  \omega ^2 \int _{ \partial \Omega} \Big | \int _{0}^{2\pi} d\theta \int_{ \delta} ^{d} \Big(\frac{ r  H_{2} ^{(1)}(r\omega) }{\omega^2}\partial_r\big (\frac{1}{r}(\partial _r f_1)\big) +\frac{i r^3  H_{2} ^{(1)}(r\omega) }{r\omega} \partial_r\big (\frac{1}{r}(\partial _rf_0 ) \big)\Big)dr     \Big|^{2}d\Gamma(x)d\omega=
 \end{equation*}
   \begin{equation*}
 \int _{k}^{\infty}  \omega ^2 \int _{ \partial \Omega} \Big | \int _{0}^{2\pi} d\theta \int_{ \delta} ^{d} \Big(\frac{ r  H_{2} ^{(1)}(r\omega) }{\omega^2}\partial_r\big (\frac{1}{r}(\partial _r f_1)\big) -\frac{i r^3}{\omega^2}H^{(1)}_3(r\omega) \partial_r (\frac{1}{r} \partial_r\big (\frac{1}{r}(\partial _rf_0 )) \big)\Big)dr     \Big|^{2}d\Gamma(x)d\omega,
 \end{equation*}
 where we used again \eqref{Hrecursion} with $\nu=3$ and the integration by parts in the term containing $i$. 
Changing back to the Cartesian coordinates and using  that $|H_{\nu} ^{(1)} (z)| \leq C$ when $1<z,\nu =1,2,3$ (due to
the inequality before \eqref{H0bound}) we yield  
\begin{equation*}
\int_{k<|\omega|} \omega ^2 \parallel u(,\omega) \parallel^2 _{(0)} (\partial \Omega) d\omega \leq
\end{equation*}
\begin{equation*}
 C \int _{k}^{\infty}  \omega^{-2} \int _{ \partial \Omega} \Big | \int _{\Omega}   \big(  \sum _{ |\alpha |\leq 2 } |\partial ^{\alpha}f_1|  + \sum _{ |\alpha |\leq 3 } |\partial ^{\alpha}f_0|\big)dy\Big|^2d\Gamma(x)d\omega \leq
\end{equation*}
\begin{equation}
\label{uomega1}
C k^{-1}\Big (  \parallel f_1  \parallel_{(2)}^{2}(\Omega)+\parallel f_0  \parallel_{(3)}^{2}(\Omega)  \Big).
\end{equation}

Now consider 
\begin{equation*}
\int_{k<|\omega|}  \parallel \nabla u(,\omega) \parallel^2 _{(0)} (\partial \Omega) d\omega = 
\end{equation*}
\begin{equation*}
=\frac{1}{16}\int_{k<|\omega|} \omega ^2 \int _{\partial \Omega } |  \int _{\Omega}   \nabla_x  H_{0} ^{(1)}(\omega|x-y|) (f_1(y)+i\omega f_0(y))dy     \Big|^{2}d\Gamma(x) d\omega. 
\end{equation*}
It is easy to see that $  \nabla_x   H_{0} ^{(1)}(\omega|x-y|)= - \nabla_y   H_{0} ^{(1)}(\omega|x-y|)$, and integrating by parts we conclude that
\begin{equation*}
\int_{k<|\omega|}  \parallel \nabla u(,\omega) \parallel^2 _{(0)} (\partial \Omega) d\omega =
\end{equation*}
\begin{equation*}
\frac{1}{16}\int_{k<|\omega|} \omega ^2 \int _{\partial \Omega } |  \int _{\Omega}     H_{0} ^{(1)}(\omega|x-y|)(\nabla_y (f_1(y)+i\omega f_0(y)))dy     \Big|^{2}d\Gamma(x) d \omega. 
\end{equation*}
Following the argument to get  \eqref{uomega1}  we complete the proof of \eqref{uomega}.

\end{proof}

The next lemma gives a bound for the initial data $f_0$ and $f_1$ by the lateral boundary data of the hyperbolic initial value problem.

\begin{lemma} 
Let $U$ be a solution to \eqref{initialhyperbolic} with $f_1 \in L^2 (\Omega), f_0 \in H^1 (\Omega) $ with $\delta<|x-y|$ when $ x\in\partial\Omega, y\in supp f_0 \cup suppf_1$ and $T = 2 diam \Omega+2$.

Then  there is $C$ such that
\begin{equation}
\label{fUbound}
\parallel f_0 \parallel _{(1)} ^2 (\Omega) +\parallel f_1 \parallel _{(0)} ^2 (\Omega)
 \leq C \Big ( \parallel \partial _t U  \parallel _{(0)} ^2 (\partial \Omega \times (0,T))+ \parallel \nabla U \parallel _{(0)} ^2 (\partial \Omega \times (0,T))      \Big )
\end{equation}
\end{lemma}

\begin{proof}
Let $U_1$ be the solution to the following initial boundary value problem 
\begin{equation*}
\partial ^{2}_{t} U_1 - \Delta U_1 =0 \quad  \text{in}  \quad \Omega \times (0,T), 
\end{equation*}
\begin{equation*}
U_1 =\partial _t U_1=0 \quad \text{on} \quad \Omega \times \{0\}, U_1 =U \quad \text{on} \quad \partial \Omega \times (0,T),
\end{equation*}
where $U$ is a solution to \eqref{initialhyperbolic}. As known \cite{LLT},
\begin{equation}
\label{fU0bound}
\parallel \partial _ {\nu} U_1 \parallel _{(0)}^{2}(\partial \Omega \times (0,T)) 
\leq C \big ( \parallel U\parallel _{(1)} ^{2} ( \partial \Omega \times (0,T)) \big).
\end{equation}
The function $U_0= U-U_1$ solves the initial boundary value problem 
\begin{equation*}
\partial ^{2}_{t} U_0 - \Delta U_0 =0 \quad  \text{in}  \quad \Omega \times (0,T), 
\end{equation*}
\begin{equation*}
U_0 =\partial _t U_0=0 \quad
\text{on}\quad \Omega \times \{0\}, U_0 =0 \quad \text{on} \quad \partial \Omega \times (0,T).
\end{equation*}
According to \cite{H}, 
\begin{equation*}
\parallel f_0 \parallel _{(1)} ^2 (\Omega) +\parallel f_1 \parallel _{(0)} ^2 (\Omega) \leq C  \parallel \partial _{\nu} U_0  \parallel _{(0)} ^2 (\partial \Omega \times (0,T)   ).
\end{equation*}
Since $U_0= U-U_1$,
\begin{equation*}
\parallel \partial _{\nu} U_0  \parallel _{(0)} ^2 (\partial \Omega \times (0,T)) \leq 2 \parallel \partial _{\nu} U  \parallel _{(0)} ^2 (\partial \Omega \times (0,T)) + 2\parallel \partial _{\nu} U_1  \parallel _{(0)} ^2 (\partial \Omega \times (0,T)) 
\end{equation*}
\begin{equation*}
\leq \parallel \partial _{\nu} U  \parallel _{(0)} ^2 (\partial \Omega \times (0,T)) + C  \parallel  U \parallel_{(1)} ^2 (\partial \Omega \times (0,T)),
\end{equation*}
due to \eqref{fU0bound}. Since $U=0$ on $\partial \Omega \times \{0\} $, by the Poincare inequality, $\parallel U \parallel_{(1)} ^2 (\partial \Omega \times (0,T))$ is bounded by the right side in \eqref{fUbound}, and we complete the proof.

                                                                                                                                                                                                        \end{proof}
\hspace{-0.5cm}\textbf{Remark 4.1.} Obviously, the following inequality holds
\begin{equation*}
\int _{k<|\omega|} \omega ^2 \parallel u(,\omega) \parallel_{(0)}^{2} (\partial \Omega)d\omega \leq k^{-2} \int _{k<|\omega|} \omega ^4  \parallel u(,\omega) \parallel_{(0)}^{2} (\partial \Omega)d\omega \leq
\end{equation*}
\begin{equation}
 k^{-2} \int _{R} \omega ^4  \parallel u(,\omega) \parallel_{(0)}^{2} (\partial \Omega)d\omega=2\pi k^{-2} \int _{R} \parallel \partial ^{2 }_{t} U(,t) \parallel_{(0)}^{2} (\partial \Omega)dt
\end{equation}
by the Parseval's identity.

Finally, we are ready to establish the  increasing stability estimate of Theorem 1.1.

\begin{proof}
Without loss of generality, we can assumed $\epsilon <1$ and $\pi (d+1)E^{-\frac{1}{4}}<\frac{1}{2}$, otherwise the bound \eqref{Istability} is straightforward. We let 
     \begin{equation}
     \label{k}
 k= K^{\frac{2}{3}}E^{\frac{1}{4}} \quad \text{if} \quad 2^{\frac{1}{4}}K^{\frac{1}{3}}< E ^{\frac{1}{4}},\quad
 \text{and}\quad
 k=K \quad \text{if}\quad E ^{\frac{1}{4}} \leq 2^{\frac{1}{4}}K^{\frac{1}{3}}.
       \end{equation} 

When $\quad E ^{\frac{1}{4}} \leq 2^{\frac{1}{4}}K^{\frac{1}{3}}$, then $k=K$ and 
\begin{equation}
\label{I11}
|I_1 (k)| \leq 2\epsilon ^2.
\end{equation}

If $2^{\frac{1}{4}}K^{\frac{1}{3}}< E ^{\frac{1}{4}}$, then from \eqref{k}, Lemma 3.2 and \eqref{I1epsilon} we obtain 
\begin{equation*}
|I_1(k)|\leq e^{2(d+1)k}e^{\frac{-2E}{\pi}\big( (\frac{k}{K})^4  -1 \big)^{\frac{-1}{2}}} CM^2\leq
\end{equation*}
\begin{equation*}
 CM^2 e^{2(d+1)K^{\frac{2}{3}}E^{\frac{1}{4}}-\frac{-2E}{\pi}(\frac{K}{k})^2 }=CM^2 e^{\frac{-2}{\pi}K^{\frac{2}{3}}E^{\frac{1}{2}}\big(1-\pi(d+1)E^{-\frac{1}{4}}\big)}.
\end{equation*}
From the inequality $e^t \leq \frac{6}{t^3}$ for $t>0$ and the  assumption at the beginning of the proof, we conclude that 
\begin{equation}
\label{I12}
|I_1(k)|\leq  CM^2  \frac{1}{K^2 E ^{\frac{3}{2}}\Big (1-\pi (d+1)E^{-\frac{1}{4}}   \Big)^3}.   
\end{equation}
Hence, using \eqref{I1}, \eqref{I1i}, \eqref{boundI1},\eqref{I11}, and \eqref{I12} we obtain 
\begin{equation*}
\int _{\partial \Omega} \int ^{+\infty}_{-\infty} \omega ^2 |u(x,\omega)|^2 d\omega d\Gamma(x)=I_1(k)+ \int _{\partial \Omega} \int _{k<|\omega|} \omega ^2 |u(x,\omega)|^2 d\omega d\Gamma(x)
\leq
\end{equation*}
\begin{equation*}
 2\epsilon ^2 + C\Big( \frac{M^2}{K^2 E ^{\frac{3}{2}}}  + \frac{\parallel f_0 \parallel_{(3)}^2 + \parallel f_1 \parallel_{(2)}^2 }{1+K^{\frac{2}{3}}E^{\frac{1}{4}}} \Big).
\end{equation*}
Similarly, by using \eqref{I2},
\eqref{I2i}, \eqref{boundI2},
we have
\begin{equation*}
\int _{\partial \Omega} \int ^{+\infty}_{-\infty}  |\nabla u(x,\omega)|^2 d\omega d\Gamma(x)
\leq 2\epsilon ^2 + C\Big( \frac{M^2}{K^2 E ^{\frac{3}{2}}}  + \frac{\parallel f_0 \parallel_{(3)}^2 + \parallel f_1 \parallel_{(2)}^2 }{1+K^{\frac{4}{3}}E^{\frac{1}{2}}} \Big).
\end{equation*}
Using these two bounds and Lemma 4.2, we finally conclude
that 
\begin{equation*}
\parallel f_1 \parallel _{(0)} ^2 (\Omega) +\parallel f_0 \parallel _{(1)} ^2 (\Omega) \leq C \Big ( \parallel \partial _t U  \parallel _{(0)} ^2 (\partial \Omega \times (0,T))+ \parallel \nabla U \parallel _{(0)} ^2 (\partial \Omega \times (0,T))\Big )\leq
\end{equation*}
\begin{equation*}
 C \Big ( \parallel \partial _t U  \parallel _{(0)} ^2 (\partial \Omega \times (0,\infty))+ \parallel \nabla U \parallel _{(0)} ^2 (\partial \Omega \times (0,\infty))      \Big )=
\end{equation*}
\begin{equation*}
C \Big (     \int _{\partial \Omega} \int ^{+\infty}_{-\infty} \omega ^2 |u(x,\omega)|^2 d\omega d\Gamma(x)+   \int _{\partial \Omega} \int ^{+\infty}_{-\infty}  |\nabla u(x,\omega)|^2 d\omega d\Gamma(x)                                      \Big ) \leq
\end{equation*}
\begin{equation*}
 2\epsilon ^2 + C\Big( \frac{M^2}{K^2 E ^{\frac{3}{2}}}  + \frac{M^2}{1+K^{\frac{4}{3}}E^{\frac{1}{2}}} \Big),
\end{equation*}
due to the Parseval's identity and the assumption that $\parallel f_0 \parallel _{(0)} ^2 (\Omega) +\parallel f_1 \parallel _{(0)} ^2 (\Omega) \leq M$, $1\leq M$. Since $ K^{\frac{2}{3}} E ^{\frac{1}{4}}<K^2 E ^{\frac{3}{2}}$ for $1<K, 1<E$, the proof is complete.
\end{proof}

\section{Conclusion}

The next analytical issue is to obtain explicit constants $C$ in the stability estimates of Theorem \ref{maintheorem} for some simple but important domains $\Omega$,
 like a sphere or a cube. This seems to be quite realistic. Another possible development is to get these estimates when $(0,K)$ is replaced
by $(K_*,K)$ with, say, $K_*=K/2$. One expects  stability results to be extended onto more general elliptic operators in any case under non trapping (pseudo convexity) conditions 
in $\Omega$. The exact controllability theory for corresponding hyperbolic equations is developed in   \cite{T}.
The needed scattering theory is  also available, although not so transparent and explicit as for the Helmholtz equation.
In this more general case it is difficult to expect constants $C$ to be explicit.
We expect that the complete Cauchy data \eqref{Green} can be replaced by, say, only the Dirichlet data when $\Gamma=\partial \Omega$. Some particular results in 
this direction are obtained in
 \cite{LY}.

There is a numerical evidence of better resolution in the inverse source problem for larger $K$ \cite{CIL}, \cite{EV}, \cite{IK}, \cite{IL}: when typically in inverse problems for elliptic equations (at low wave numbers) with a realistic noise in the data one can
recover 5-8 parameters of 
unknown source, using wave numbers up to 100 one can
achieve a stable reconstruction of 30-50 parameters.
 It is important to collect further numerical evidence of  the increasing stability, in particular for more complicated geometries.
 
The numerical examples in \cite{IL1} suggest that increasing stability is possible without any (pseudo)convexity condition. This is a very interesting and
seemingly hard topic to investigate. Some results on the increasing stability of the continuation without convexity assumptions are obtained in \cite{I} and \cite{IK}.

It is challenging and important to consider recovery of the source term for the equation
\begin{equation*} 
(\Delta +k^2)u = -f_1 -ikf_0 \quad \text{in}\quad \Omega
\end{equation*}
in a bounded domain $\Omega$ when the radiation condition
\eqref{radiation} is replaced by one of boundary conditions
on $\partial\Omega$, like
\begin{equation*} 
\partial_{\nu}u = 0 \quad \text{on}\quad \partial\Omega.
\end{equation*}
The additional data can be
$u$ on $\Gamma\subset \partial
\Omega, K_*<k<K$. The solution $u$ is then a meromorphic function of $k$
(with poles at eigenvalues of the corresponding elliptic boundary value problem) and  a priori bounds for the direct problem are getting more complicated, but we expect that the main ideas can be properly adjusted to this case.

{\bf Acknowledgment:} 

 This research is supported in part by the Emylou Keith and Betty Dutcher Distinguished Professorship and the NSF grant DMS 15-14886.


\begin{thebibliography}{99}


\bibitem{ABF} Ammari H,  Bao G and Fleming J 2002
 Inverse source problem for Maxwell's equation in magnetoencephalography {\it SIAM J. Appl. Math.} {\bf 62}  1369-82.

\bibitem{B} Balanis C 2005 {\it Antenna Theory- Analysis and Design} (Wiley, Hoboken, NJ)

\bibitem{BJT} Bao G, Lin J and Triki F 2010 
 A multi- frequency inverse source problem {\it J. Diff. Equat.} {\bf 249}  3443-3465.

\bibitem{CIL} Cheng J,  Isakov V and Lu S 2016
Increasing stability in the inverse source problem with many frequencies, {\it J.Diff. Equat.} {\bf 260} 4786-4804.


\bibitem{EV}
Eller M and Valdivia N 2009 Acoustic source identification using multiple frequency information {\it Inverse Problems} {\bf 25} 115005
 

\bibitem{H}  Ho L F 1986 Observabilite frontiere de l'equation des ondes {\it C. R. Acad. Sc. Paris} {\bf 302} 443-446.

\bibitem{I} Isakov V 2017
{\it Inverse Problems for Partial Differential Equations} ( Springer-Verlag, New York)

\bibitem{IK}
Isakov V and Kindermann S 2011 Regions of stability in the Cauchy problem
  for the Helmholtz equation {\it Methods Appl. Anal.} {\bf 18} 1--30

\bibitem{IL} Isakov V and S. Lu S Increasing stability in the inverse source  problem with attenuation and many frequencies {\it SIAM J. Appl. Math} (to appear ).


\bibitem{IL1} Isakov V and Lu S, Increasing stability in the inverse source  problem with attenuation and many frequencies {\it Inverse Problems Imaging} ( submitted ).

\bibitem{J} John F 1982 {\it Partial Differential Equations} (Applied Mathematical Sciences, Springer-Verlag, New York, Berlin)

\bibitem{LLT} Lasiecka I Lions J L and Triggiani R 1986 Non Homogeneous Boundary Value Problems for Second Order Hyperbolic Operators,
{\it J. Math. pures appl.}  {\bf 65}  149-192

\bibitem{LY} Li P and Yuan G 2017  
Increasing stability for the inverse source scattering problem with multi-frequencies 
{\it Inverse Problems and Imaging} {\bf 11}  745-759
  

\bibitem{T} Tataru D 1995
 Sharp sufficient conditions for the observation, control, and stabilization of waves from the boundaries {\it Comm. Part. Diff. Equat.}
  {\bf 20}  855--884
                                                                                                                                                                                                                                                                                                                                                                                                                                                                                                                                                                                                                                                                                                                                                                                                                                                                                                                                                                                                                       

\bibitem{W}  Watson G N 1922
{\it  A Treatise on the Theory of Bessel Functions} (Cambridge University Press)

\end{thebibliography}
\end{document}